%% file: AHO-GRP.tex
\documentclass[12pt,final]{amsart}

\usepackage{graphicx}
\usepackage{psfrag}
\usepackage{subfigure}
\usepackage{color}
\usepackage{bm}

\usepackage{showkeys}

\usepackage{paralist}
\usepackage[colorlinks=true]{hyperref}

\numberwithin{equation}{section} \numberwithin{figure}{section}
\def\n{\noindent}
\def\pt{\partial}

\def\f#1#2{\frac {#1}{#2}}
\def\f32{\frac 32}

\def\d{\displaystyle}

\def\b#1{\overline#1}
\def\bga{\begin{array}}
\def\eda{\end{array}}

\def\De{\Delta}

\def\ev{\equiv}

\def\Gm{\Gamma}
\def\gm{\gamma}

\def\nb{\nabla}

\def\la{\lambda}
\def\La{\Lambda}

\def\al{\alpha}
\def\td{\tilde}
\def\rw{\rightarrow}
\def\iy{\infty}
\def\Om{\Omega}
\def\om{\omega}

\def\d{\displaystyle}
\def\dfr#1#2{\displaystyle{\frac{#1}{#2}}}

 \newtheorem{thm}{Theorem}[section]

 \newtheorem{prop}[thm]{Proposition}
 \theoremstyle{definition}
 
 \theoremstyle{remark}
 \newtheorem{rem}[thm]{Remark}
 
\include{abbr-notation}

\title{A Two-Stage  Fourth Order Time-Accurate Discretization for   Lax-Wendroff Type Flow  Solvers \\[3mm]
I.  Hyperbolic Conservation Laws }

\author{Jiequan Li and Zhifang Du}

\address{Jiequan Li: Laboratory of Computational Physics,  Institute of Applied Physics and Computational Mathematics, Beijing, P. R. China; Email: li\_jiequan@iapcm.ac.cn}

\address{Zhifang Du: School of Mathematical Sciences, Beijing Normal University, 100875, P. R. China; Email: du@mail.bnu.edu.cn}

\begin{document}
\maketitle
\markboth{Jiequan Li and Zhifang Du}{A Two-Stage  Fourth Order  L-W  Type  Scheme  }

\begin{abstract} In this paper we develop a novel two-stage  fourth order time-accurate  discretization  for time-dependent flow problems, particularly for  hyperbolic conservation laws. Different from the classical Runge-Kutta (R-K)  temporal discretization for   first order Riemann solvers as building blocks, 
the current approach   is solely associated with Lax-Wendroff (L-W)  type  schemes as the  building blocks. As a result,  a two-stage procedure can be constructed to achieve a  fourth order temporal accuracy, rather than using well-developed four stages for R-K methods. The generalized Riemann problem (GRP) solver is taken   as a representative of L-W type schemes for the construction of a two-stage fourth order scheme. 
\end{abstract} 



 \section{Introduction} 
 
The design of high order accurate  CFD methods has attracted much attention  in the past decades. Successful examples include ENO \cite{Harten-ENO, Shu-Osher,Barth}, WENO \cite{Liu, Jiang}, DG \cite{DG-2},
residual distribution (RD) method \cite{Abgrall-RD},  spectral methods \cite{Tang-Shen} etc., and references therein. Most of these methods use the Runge-Kutta (R-K) approach to achieve high order temporal accuracy   starting from the first order numerical flux functions, such as  first order Riemann solvers. In order to  achieve a  fourth order temporal accuracy,   four stages of R-K type iterations in time are usually adopted.  
\vspace{0.2cm} 

In this paper we develop a novel fourth order temporal discretization  for time-dependent problems, particularly for  hyperbolic conservation laws 
\begin{equation}
\dfr{\pt \bu}{\pt t} +\nb\cdot \bbf(\bu)=0,
\label{law}
\end{equation}
where $\bu=(u_1,\cdots, u_m)^\top$ is a conservative vector, $\bbf(\bu)=(\bbf_1(\bu),\cdots, \bbf_d(\bu))$ is the associated flux vector function, $m\geq 1$, $d\geq 1$. The approach under investigation is based on the second order Lax-Wendroff (L-W)  methodology and uses a two-stage procedure to achieve  a  fourth order accuracy,  which is different from  the classical R-K approach.  This approach can be easily extended to many other time-dependent flow problems \cite{Du-Li}. 

\vspace{0.2cm} 

The Lax-Wendroff methodology \cite{L-W}, i.e.,    the Cauchy-Kovalevskaya method in the context of PDEs,  is  fundamental in the sense that it has  second order accuracy  both in space and time,  and the underlying governing equations are fully incorporated into approximations of spatial and temporal evolution.  In a finite volume framework,   Eq.  \eqref{law} is discretized as 
\begin{equation}
\begin{array}{l}
\d \b\bu_j^{n+1} =\b\bu_j^n -\sum_{\ell}\frac{\De t}{|\Om_j|}\bF_{j\ell}(\bu(\cdot, t_n+\frac{\De t}2),\Gm_{j\ell},\bn_{j\ell})
\label{scheme}
\end{array}
\end{equation} 
where  $\b\bu_j^n$ is the solution averaged over the control volume $\Om_j$  at time $t=t_n$,  $t_{n+1}=t_n+\De t$, $\Gm_{j\ell}$ is the $\ell$-th side of $\Om_j$,  and $\bn_{j\ell}$ is the unit outward normal direction.  The numerical flux $\bF_{j\ell}$ along the time-space surface $\Gm_{j\ell}$ is based on the half time-step value $\bu(\cdot, t+\De t/2)$,  or an equivalent approximation. The L-W method achieves the time accuracy  through the formulae
\begin{equation}
\begin{array}{l}
\d \bu(\bx, t_n+\frac{\De t}2) =\bu(\bx, t_n) +\dfr{\De t}{2} \cdot \dfr{\pt \bu}{\pt t}(\bx,t_n) +\mathcal{O}(\De t^2), \ \ \ \bx\in \Gm_{j\ell}, \\
\dfr{\pt \bu}{\pt t}(\bx,t_n) =-\nb\cdot \bbf(\bu)(\bx, t_n), 
\end{array}
\end{equation}
which  adopt two instantaneous values $\bu(\bx, t_n)$ and $\frac{\pt \bu}{\pt t}(\bx, t_n)$ at any point $(\bx, t_n)$ on the boundary of a control volume. In particular, the time variation is related to the spatial derivatives of the solutions. The   first order {\em Riemann solvers}  \cite{Godunov, Toro} and second order {\em  L-W solvers}  have distinguishable  procedures to  define the two instantaneous values, respectively.   The L-W method is a one-stage spatial-temporal coupled second order accurate method, which  utilizes  the information only at time $t=t_n$.  If a R-K  approach is preferred, usually two stages are needed to achieve a second order accuracy in time.

\vspace{0.2cm}

Our goal is to extend the L-W type schemes to even higher order accuracy.     Based on  a second order L-W type solver with the information $\bu$ and $\frac{\pt \bu}{\pt t}$,  a two-stage procedure can be designed  to obtain a  fourth order temporal accurate approximation for $\bu(\cdot, t_{n+1})$: one stage at $t=t_n$ and the other stage at $t_n+\frac{\De t}2$. The algorithm is stated below.   \vspace{0.2cm} 

\begin{enumerate}
{\em 
\item[(i)] {\bf Lax-Wendroff step}. Given an initial data $\bu^n(x)$ to \eqref{law} at $t=t_n$, construct instantaneous values $\bu(\bx,t_n+0)$ and $\frac{\pt\bu}{\pt t}(\bx,t_n+0)$, which are symbolically denoted as 
\begin{equation}
\bu(\cdot, t_n+0) = \mathcal{M}(\bu^n),\ \ \ \ \frac{\pt}{\pt t}\bu(\cdot,t_n+0) = \mathcal{L}(\bu^n). 
\end{equation}
Then $\frac{\pt}{\pt t}\mathcal{L}(\bu)(\cdot,t_n+0)$ is subsequently obtained using the chain rule, 
\begin{equation}
\dfr{\pt}{\pt  t} \mathcal{L}(\bu^n)=\dfr{\pt}{\pt\bu}\mathcal{L}(\bu^n) \dfr{\pt}{\pt t}\bu(\cdot,t_n+0).
\end{equation}
\vspace{0.2cm}

\item[(ii)]  {\bf Solution advancing step}. Define the intermediate  data $\bu^*(\bx)$
\begin{equation}
\bu^* =\bu^n +\dfr 12 \De t\mathcal{L}(\bu^n) +\dfr 18\De t^2 \dfr{\pt }{\pt t}\mathcal{L}(\bu^n),
\end{equation} 
which can be  used to reconstruct new initial data $\bu^*(\bx)$ and get the solution $\frac{\pt }{\pt t}\mathcal{L}(\bu^*)$.
 Then  the solution to the next time level $t_{n+1} =t_n+\De t$ can be updated by
\begin{equation}
\begin{array}{l}
\bu^{n+1} =\bu^n + \De t \mathcal{L}(\bu^n) + \dfr 16 \De t^2 \left(\dfr{\pt }{\pt t}\mathcal{L}(\bu^n)+2\dfr{\pt }{\pt t}\mathcal{L}(\bu^*)\right). 
\end{array} 
\end{equation}
} 
\end{enumerate} 
\vspace{0.2cm} 
 
The above updating scheme distinguishes    from the traditional R-K approach in the following aspects.
\vspace{0.2cm}

(i) The above approach is based on a second order L-W type solver to achieve a fourth order temporal accuracy, which is different from  traditional R-K methods for \eqref{law}  based on first order solvers. It is a two-stage approach, while the R-K approach usually needs four stages to attain the same accuracy. With the data reconstruction in  the middle  stage,  this new approach removes  two stages of data reconstruction  from the standard R-K approach.  Together with numerical flux evaluation,  at least about $20\%$ computational cost can be saved for 1-D problems,  and $30\%$ cost saved for 2-D problems in the current method. 

\vspace{0.2cm} 

(ii) The governing equation \eqref{law} is explicitly used in the L-W solver so that all useful information can be included in the solution approximation.  More importantly, we stick to the utilization of the time derivatives of solutions $\pt \bu/\pt t$ to advance the solution, which seems more effective to  capture discontinuities sharply.   See Remark \ref{rem-grp} in Section \ref{sec-law} for the explanation. \vspace{0.2cm}

(iii)   This approach can be applied in other frameworks,  such as  DG or finite difference  methods etc.  Not only for hyperbolic conservation laws \eqref{law}, other time-dependent problems can be solved by the above approach as well once there is a corresponding   Cauchy-Kovalevskaya theorem.   

\vspace{0.2cm} 

Since this method is based on L-W type solvers, the generalized Riemann problem (GRP) solver \cite{Ben-Artzi-84, Ben-Artzi-01, Li-1, Li-2} is used  as the building block. The GRP solver is an  extension of the first order Godunov solver to the second order time accuracy from the  MUSCL-type initial data \cite{Leer}.  Its simplified acoustic version reduces to the so-called ADER solver \cite{Toro-ADER}, which can  of course be used as the building block too.  Other alternative choice could be the gas kinetic solvers (GKS) \cite{Xu-1, Xu-2}. Numerical experiments from  the GRP solver demonstrate the suitability to design such a fourth order method.  
\vspace{0.2cm}

This paper is organized in the following. After the introduction section, a two-stage temporal discretization is presented in Section 2.  In Section 3,  this approach is applied to hyperbolic conservation laws in 1-D and 2-D, respectively. 
 In Section 4,  numerical experiments for   scalar conservation laws and the compressible Euler equations are taken to validate the performance of the proposed approach.  The last section presents  discussions and some prospectives of this approach.

\vspace{0.2cm}

 \section{A  high order temporal discretization for time-dependent problems} 

In order  to advance the solution of \eqref{law} with a  fourth order temporal accuracy  for the  L-W type solvers,  consider the following time-dependent  equations,
\begin{equation}
\dfr{\pt \bu}{\pt t}=\mathcal{L}(\bu),
\label{ODE} 
\end{equation} 
subject to the initial data at $t=t_n$, 
\begin{equation}
\bu(t)|_{t=t_n}=\bu^n, 
\end{equation}
where $\mathcal{L}$ is an operator for  spatial derivatives. 
It is evident that the initial time variation of the solution at $t=t_n$ can be obtained using the chain rule and the Cauchy-Kovalevskaya method,
\begin{equation}
\dfr{\pt}{\pt t}\bu(t_n) =\mathcal{L}(\bu^n), \ \ \ \ \dfr{\pt^2}{\pt t^2}\bu(t_n) =\dfr{\pt}{\pt t}\mathcal{L}(\bu^n) = \dfr{\pt}{\pt \bu}\mathcal{L}(\bu^n)\mathcal{L}(\bu^n).
\end{equation} 
Let's consider a high order accurate approximation to $\bu^{n+1}:= \bu(t_n+\De t)$.  We write \eqref{ODE} as
\begin{equation}
\bu^{n+1} =\bu^n +\int_{t_n}^{t_n+\De t}  \mathcal{L}(\bu(t))dt. 
\label{ODE-A}
\end{equation}
Introduce an intermediate value at time $t=t_n+ A\De t$  with a parameter $A$, within a third order accuracy, 
\begin{equation}
\bu^* =  \bu^n +A\De t \mathcal{L}(\bu^n) +\dfr 12 A^2 \De t^2 \dfr{\pt}{\pt t}\mathcal{L}(\bu^n),
\label{appr-1}
\end{equation} 
which  subsequently determines the solution at the middle stage,
\begin{equation}
\dfr{\pt \bu^*}{\pt t}  =\mathcal{L}(\bu^*), \ \ \ \ \dfr{\pt}{\pt t}\mathcal{L}(\bu^*) =\dfr{\pt}{\pt \bu}\mathcal{L}(\bu^*) \mathcal{L}(\bu^*). 
\label{appr-2}
\end{equation}
Set 
\begin{equation}
\bu^{n+1}=\bu^n + \De t(B_0 \mathcal{L}(\bu^n) + B_1 \mathcal{L}(\bu^*)) +\frac 12 \De t^2 \left(C_0 \frac{\pt}{\pt t} \mathcal{L}(\bu^n) + C_1 \frac{\pt}{\pt t}\mathcal{L}(\bu^*)\right),
\label{appr-3} 
\end{equation}
where $B_0$, $B_1$, $C_0$ and $C_1$, together with $A$,  are  determined according to accuracy requirement. We formulate this approximation in the form of a proposition. 

\vspace{0.2cm} 
\begin{prop}  If  the following parameters are taken,
\begin{equation}
A=\frac 12, \ \ \  B_0= 1, \ \ \ B_1=0, \ \ \ C_0 =\frac 13, \ \ \ \ C_1=\frac 23,
\label{coeff-ode} 
\end{equation} 
the iterations \eqref{appr-1}--\eqref{appr-3} provide a fourth order accurate approximation to the solution $\bu(t)$ at $t=t_n+\De t$. These parameters are uniquely determined for the fourth order accuracy requirement. 
\end{prop} 

\begin{proof} The proof uses the standard Taylor series expansion, as usually done for the R-K approach.   For notational simplicity, we denote 
\begin{equation}
\mathcal{G}(\bu) := \mathcal{L}_\bu(\bu)\mathcal{L}(\bu), \ \ \ \ \  \ \ L_\bu(\bu):=\dfr{\pt}{\pt \bu}\mathcal{L}(\bu),
\end{equation} 
and similarly for $\mathcal{L}_{\bu\bu}$, $\mathcal{L}_{\bu\bu\bu}$, $\mathcal{G}_\bu$, and $\mathcal{G}_{\bu\bu}$. 
Then we have the following expansions around $\bu^n$,
\begin{equation}
\mathcal{L}(\bu^*) =\mathcal{L}(\bu^n) + \mathcal{L}_\bu (\bu^*-\bu^n) +\frac{\mathcal{L}_{\bu\bu}}2 (\bu^*-\bu^n)^2 + \frac{\mathcal{L}_{\bu\bu\bu}}6 (\bu^*-\bu^n) ^3+\mathcal{O}(\bu^*-\bu^n)^4, 
\end{equation}
and 
\begin{equation}
\mathcal{G}(\bu^*) =\mathcal{G}(\bu^n) + \mathcal{G}_\bu (\bu^*-\bu^n) +\frac{\mathcal{G}_{\bu\bu}}2 (\bu^*-\bu^n)^2  +\mathcal{O}(\bu^*-\bu^n)^3.
\end{equation}
Using \eqref{appr-1} and \eqref{appr-2}, as well as substituting the above two expansions into \eqref{appr-3}, we obtain
\begin{equation} 
\begin{array}{rl}
& \De t(B_0 \mathcal{L}(\bu^n) + B_1 \mathcal{L}(\bu^*)) +\frac 12 \De t^2 \left(C_0 \frac{\pt}{\pt t} \mathcal{L}(\bu^n) + C_1 \frac{\pt}{\pt t}\mathcal{L}(\bu^*)\right)\\[3mm]
=& \De t(B_0+B_1)\mathcal{L}(\bu^n)+ \dfr{\De t^2}{2} \left[AB_1+\frac 12(C_0+C_1)\right]\mathcal{L}_\bu(\bu^n)\mathcal{L}(\bu^n)\\[3mm] 
&\d  +\frac{\De t^3}6\left[3(A^2B_1 +AC_1)\right]\cdot [\mathcal{L}_{\bu}^2(\bu^n)\mathcal{L}(\bu^n) +\mathcal{L}_{\bu\bu}(\bu^n)\mathcal{L}^2(\bu^n)]\\[3mm]
&\d +\frac{\De t^4}{24} \left[6A^2 C_1 \mathcal{L}_\bu^3(\bu^n)\mathcal{L}(\bu^n) + (12A^3 B_1+24A^2 C_1)(\mathcal{L}_{\bu\bu}(\bu^n)\mathcal{L}_\bu(\bu^n)\mathcal{L}^2(\bu^n)\right. \\[3mm]
&\d +\left.(4A^3 B_1+ 6A^2 C_1)\mathcal{L}_{\bu\bu\bu}(\bu^n)\mathcal{L}(\bu^n)\right] +\mathcal{O}(\De t^5). 
\end{array}
\label{t-1} 
\end{equation}
Taking the Taylor series expansion directly for the time integration in \eqref{ODE-A} yields 
\begin{equation}
\begin{array}{rl}
&\d \int_{t_n}^{t_n+\De t}  \mathcal{L}(\bu(t))dt\\[3mm] &= \d \De t\mathcal{L}(\bu^n) +\dfr{\De t^2}2 \mathcal{L}_\bu(\bu^n)\mathcal{L}(\bu^n)+\dfr{\De t^3}6[\mathcal{L}_\bu^2(\bu^n) \mathcal{L}(\bu^n) +\mathcal{L}_{\bu\bu}(\bu^n) \mathcal{L}^2(\bu^n)] \\[3mm] 
&+ \dfr{\De t^4}{24} [\mathcal{L}_\bu^3(\bu^n)\mathcal{L}(\bu^n) +4 \mathcal{L}_{\bu\bu}(\bu^n)\mathcal{L}_\bu(\bu^n)\mathcal{L}^2(\bu^n) +\mathcal{L}_{\bu\bu\bu}(\bu^n)\mathcal{L}(\bu^n)]+\mathcal{O}(\De t^5).
\end{array}
\label{t-2}
\end{equation} 
The  comparison of  \eqref{t-1} and \eqref{t-2} gives
\begin{equation}
\begin{array}{c}
B_0+B_1 =1, \ \ \  AB_1 +\frac 12(C_0+C_1) =\frac 12, \ \ \  \ A^2 B_1 +AC_1 =\frac 13, \\[2mm] 
A^2C_1=\frac 16, \ \ \  A^3 B_1 +2A^2C_1 = \frac 13, \  \ \ 
2A^3 B_1 +3A^2 C_1 =\frac 12. 
\end{array}
\end{equation}
The above equations uniquely determine   $A$, $B_0$, $B_1$, $C_0$ and $C_1$ with the values in \eqref{coeff-ode}. 
\end{proof} 

Thus we present the algorithm for \eqref{ODE}  explicitly.

\vspace{0.2cm} 
\n{\bf Algorithm-general.}
\begin{enumerate}
\item[\bf Step 1.] Define intermediate values 
\begin{equation}
\begin{array}{l}
\bu^* =  \bu^n +\frac 12\De t \mathcal{L}(\bu^n) +\dfr 18 \De t^2 \dfr{\pt}{\pt t}\mathcal{L}(\bu^n),\\
\dfr{\pt}{\pt t}\mathcal{L}(\bu^*) =\dfr{\pt}{\pt \bu}\mathcal{L}(\bu^*) \mathcal{L}(\bu^*).
\end{array}
\label{a-1}
\end{equation}

\item[\bf Step 2.]  Advance the solution using the formula
\begin{equation}
\bu^{n+1} =\bu^n + \De t \mathcal{L}(\bu^n) + \dfr 16 \De t^2 \left(\dfr{\pt }{\pt t}\mathcal{L}(\bu^n)+2\dfr{\pt }{\pt t}\mathcal{L}(\bu^*)\right). 
\end{equation}
\label{a-2}
\end{enumerate}

 \vspace{0.2cm}
 
\begin{rem}  Note that $A=\frac 12$. The time $t_n+A\De t$ is in the middle  of  interval $[t_n, t_n+\De t]$ and $\bu^*$ is the mid-point value of $\bu$.  Therefore, the iterations \eqref{appr-1}--\eqref{appr-3} can be actually regarded as  an Hermite-type approximation to \eqref{ODE-A}. In contrast, the classical R-K iteration method  is written as 
\begin{equation}
\begin{array}{l} 
\d \bu^{(i)} = \sum_{k=0}^{i-1}  \al_{ki}  \bu^{(k)} + \De t\beta_{ki}\mathcal{L}(\bu^{(k)}), \ \ \ \ i = 1, ... ,m, \\
\bu^{(0)} =\bu^n, \ \ \ \ \ \bu^{(m)} =\bu^{n+1}, 
\end{array} 
\label{R-K}
\end{equation}
where $\al_{ki}\geq 0$, $\beta_{ki}>0$ are the integration weights,  satisfying the compatibility condition $\d \sum_{k=0}^{i-1} \al_{ki} =1$. 
 Since the approximation \eqref{R-K}  does not involve the derivative of $\mathcal{L}$, it is regarded as the Simpson-type approximation to \eqref{ODE-A}. 

 \vspace{0.2cm}

\end{rem}
  
\vspace{0.2cm}
 
  \section{Fourth order accurate temporal discretization for  hyperbolic conservation laws}\label{sec-law}
  
 In this section, we will extend the approach in the last section to hyperbolic conservation laws \eqref{law}  to design a time-space fourth order accurate method. This extension is based on  L-W type solvers with   the instantaneous   solution  and its temporal derivative through the governing equation \eqref{law}. 
 We will first discuss the one-dimensional case,  and then go to the two-dimensional case. 
 \vspace{0.2cm}
 
\subsection{One-dimensional hyperbolic conservation laws}  Let us start with  one-dimensional hyperbolic conservation laws 
\begin{equation}
\bu_t +\bbf(\bu)_x=0,
\label{1d-law} 
\end{equation}
where $\bu$ is, as in \eqref{law}, a conserved variable and $\bbf(\bu)$ is the associated flux function. 
We integrate it over the cell  $(x_{j-\frac 12},x_{j+\frac 12})$ to obtain the semi-discrete form
\begin{equation}
\dfr{d}{dt} \bar \bu_j(t) =\mathcal{L}_j(\bu): =-\dfr{1}{\De x_j} [\bbf_{j+\frac 12}-\bbf_{j-\frac 12}],  
\label{semi} 
\end{equation} 
where $\bbf_{j+\frac 12}$ is the numerical flux through the cell boundary $x=x_{j+\frac 12}$ at time $t$, $\De x_j=x_{j+\frac 12}-x_{j-\frac 12}$. We construct initial data for \eqref{semi} through a  fifth order WENO or HWENO interpolation technology \cite{Jiang, Qiu},
\begin{equation}
\bu(x,t_n) =\bu^n(x). 
\label{data-weno} 
\end{equation}   
Based on this initial condition, with possible  discontinuities at the cell boundaries,  the instantaneous  solution can be obtained,
\begin{equation}
\bu_{j+\frac 12}^n : = \lim_{t\rw t_n+0} \bu(x_{j+\frac 12},t), \ \ \ \ \left(\dfr{\pt\bu}{\pt t}\right)_{j+\frac 12}^n : =  \lim_{t\rw t_n+0} \dfr{\pt}{\pt t}\bu(x_{j+\frac 12},t). 
\label{1d-grp}
\end{equation} 
There  is an analytical solution for   the generalized Riemann problem (GRP) solver \cite{Li-1,Li-2}; or approximately as  ADER solvers \cite{Toro-ADER}. Intrinsically, the temporal derivative $(\pt\bu/\pt t)_{j+\frac 12}^n$ is  replaced by the spatial derivative at time $t=t_n$ using the governing equation \eqref{1d-law},
\begin{equation}
 \left(\dfr{\pt\bu}{\pt t}\right)_{j+\frac 12}^n =-\lim_{t\rw t_n+0} \dfr{\pt}{\pt x}\bbf(\bu(x_{j+\frac 12},t)),
\end{equation} 
where the spatial derivative takes account of the wave propagation. This approach is called the L-W approach numerically or the Cauchy-Kovalevskaya approach in the context of PDE theory.  In the numerical experiments in Section 4, we use the GRP solver developed in \cite{Li-1,Li-2}  and construct the corresponding algorithm for \eqref{1d-law}. This two-stage approach for \eqref{1d-law} is proposed as follows.
\vspace{0.2cm}

\n{\bf Algorithm  1-D.} 
\begin{enumerate}

\item[\bf Step 1.]  With the initial data $\bu^n(x)$ in \eqref{data-weno} obtained by the HWENO interpolation, we compute the instantaneous values $\bu_{j+\frac 12}^n$ and $(\pt\bu/\pt t)_{j+\frac 12}^n$ analytically or approximately using a L-W type solver. 
\vspace{0.2cm} 

\item[\bf Step 2.]  Construct the intermediate values $\bu^*(x)$  at $t_*=t_n+\frac 12\De t$ using the formulae,
\begin{equation}
\begin{array}{l}
\bar\bu_j^* =\bar\bu_j^n -\dfr{\De t}{2\De x_j}[\bbf_{j+\frac 12}^*-\bbf_{j-\frac 12}^*],  \\[3mm] 
\d \bbf_{j+\frac 12}^* = \bbf(\bu(x_{j+\frac 12},t_n+\frac 14\De t)), \\[3mm]
 \d  \bu(x_{j+\frac 12},t_n+\frac 12\De t):= \bu_{j+\frac 12}^n +\frac{\De t}{2}\left(\frac{\pt\bu}{\pt t}\right)_{j+\frac 12}^n.
\end{array}
\label{1d-al}
\end{equation} 
Then we use the HWENO interpolation again to construct $\bu^*(x)$ and find the values $\bu_{j+\frac 12}^*$ and $(\pt\bu/\pt t)_{j+\frac 12}^*$ at stage $t=t_n+\frac{\De t}2$, as done in Step 1. 
\vspace{0.2cm} 

\item[\bf Step 3.]  Advance the solution to the next time level $t_n+\De t$, 
\begin{equation}
\begin{array}{l}
\bu^{n+1}_j =\bar \bu_j^n -\dfr{\De t}{\De x_j} [\bbf_{j+\frac 12}^{4th} -\bbf_{j-\frac 12}^{4th}],\\[3mm]
\d \bbf_{j+\frac 12}^{4th}=\bbf(\bu_{j+\frac 12}^n) +\dfr{\De t}{6} \left[\left.\dfr{\pt  \bbf(\bu)}{\pt t} \right|_{(x_{j+\frac 12}, t_n) } +2\left. \dfr{\pt  \bbf(\bu)}{\pt t}\right|_{ (x_{j+\frac 12}, t_*)}\right]. \\[3mm] 
 
\end{array}
\label{1d-al2}
\end{equation} 

\end{enumerate}

\vspace{0.2cm} 

This is exactly a two-stage method: One stage at $t=t_n$ and the other at $t=t_n+\frac{\De t}2$.   We only need to reconstruct  data and use the L-W solver twice, at $t=t_n$ and $t=t_n+\frac 12 \De t$, respectively.  The procedure to reconstruct the intermediate state $\bu^*(x)$ and get  the GRP solution at time $t=t_n+\frac{\De t}2$  is the same as that at time  $t=t_n$. 
 \vspace{0.2cm} 

\begin{rem} \label{rem-grp}  The utilization of the time derivative $(\pt\bu/\pt t)_{j+\frac 12}^n$  is one of central points in our algorithm. Indeed,  
the fully explicit form of  \eqref{1d-law} is, 
\begin{equation}
\bar\bu_j^{n+1} =\b\bu_j^n-\dfr{\De t}{\De x_j}\left[\dfr{1}{\De t}\int_{t_n}^{t_{n+1} }\bbf(\bu(x_{j+\frac 12},t))dt-\dfr{1}{\De t} \int_{t_n}^{t_{n+1} }\bbf(\bu(x_{j-\frac 12},t))dt\right]. 
\end{equation} 
It is crucial to approximate the flux at $x=x_{j+\frac 12}$ in the sense that
\begin{equation}
\mbox{Numerical flux at } x_{j+\frac 12} - \dfr{1}{\De t}\int_{t_n}^{t_{n+1} }\bbf(\bu(x_{j+\frac 12},t))dt =\mathcal{O}(\De t^r), \ \ r>1. 
\end{equation}
Many algorithms approximate the flux with error measured by $\De\bu$, the jump across the interface,
\begin{equation}
\mbox{Numerical flux at } x_{j+\frac 12} - \dfr{1}{\De t}\int_{t_n}^{t_{n+1} }\bbf(\bu(x_{j+\frac 12},t))dt =\mathcal{O}(\|\De \bu\|^r), 
\end{equation}
 which is not proportional to the mesh size $\De x_j$ or the time step length $\De t$ when the jump is large, e.g., strong shocks,
 \begin{equation}
 \|\De \bu\| \not\approx\mathcal{O}(\De x_j). 
 \end{equation}
 It turns out that there is a large discrepancy  when strong discontinuities  present in the solutions.  In order to overcome this difficulty, we have to solve the associated generalized Riemann problem (GRP) analytically and derive the value $(\pt\bu/\pt t)_{j+\frac 12}^n$ and subsequently $(\pt\bu/\pt t)_{j+\frac 12}^*$. 
\end{rem}

\begin{rem} Without using the data reconstruction for $\bu^*(x)$, the above procedure could be regarded as the Hermite-type approximation to the total flux in the sense 
 \begin{equation}
\dfr{1}{\De t}\int_{t_n}^{t_{n+1}}   \bbf(\bu(x_{j+\frac 12},t)) dt =\sum_{k=1}^2 \left[C_k \bbf(\bu(x_{j+\frac 12},t_n+\al_k\De t)) +\De t D_k \dfr{\pt \bbf}{\pt t} (\bu(x_{j+\frac 12},t_n+\al_k\De t))\right],
\label{qua}
\end{equation} 
where $t_n+\al_k\De t$ are the quadrature nodes, $C_k$, $D_k$ are the quadrature weights.   For linear equations,  the  formula is exact. We can further verify through numerical examples that this two-stage method indeed provides a  temporal discretization with fourth order accuracy. 

\end{rem}
\vspace{0.2cm} 

\subsection{Multidimensional hyperbolic conservation laws.} 

For multidimensional cases of \eqref{law},  we  still use the finite volume framework  to develop a two-stage temporal-spatial fourth order accurate method. For simplicity of presentation, we only consider the  two-dimensional (2-D) case with  rectangular meshes in the present paper. All other cases can be treated analogously, e.g., over unstructured meshes. 
\vspace{0.2cm}

We write the 2-D case of \eqref{law} as
\begin{equation}
\dfr{\pt \bu}{\pt t} +\dfr{\pt \bbf(\bu)}{\pt x} +\dfr{\pt \bg(\bu)}{\pt y}=0.
\label{2D-law} 
\end{equation} 
The computational domain $\Om$ is divided into rectangular meshes $K_{ij}$, $\Om=\cup_{i\in I, j\in J} K_{ij}$,  $K_{ij}=(x_{i-\frac 12},x_{j+\frac 12})\times (y_{j-\frac 12},y_{j+\frac 12})$ with $(x_i,y_j)$ as the center. Then \eqref{2D-law} reads over $K_{ij}$ 
\begin{equation}
\dfr{d\bar\bu_{i,j}(t)}{dt} =\mathcal{L}_{i,j}(\bu): =-\dfr{1}{\De x_i}[\bbf_{i+\frac 12, j} -\bbf_{i-\frac 12,j}] -\dfr{1}{\De y_j} [ \bg_{i,j+\frac 12}-\bg_{i,j-\frac 12}], 
\end{equation} 
where $\De x_i=x_{i+\frac 12}-x_{i-\frac 12}$, $\De y_j=y_{j+\frac 12}-y_{j-\frac 12}$, and as convention, 
\begin{equation}
\begin{array}{l}
\d \bar\bu_{i,j} =\dfr{1}{\De x_i\De y_j}\int_{K_{ij}} \bu(x,y,t)dxdy, \\[3mm]  
\d \bbf_{i+\frac 12,j} (t)=\dfr{1}{\De y_j}\int_{y_{j-\frac 12}}^{y_{j+\frac 12}} \bbf(\bu(x_{i+\frac 12},y,t))dy, \\[3mm] 
   \bg_{i,j+\frac 12}(t) =\dfr{1}{\De x_i} \int_{x_{i-\frac 12}}^{x_{i+\frac 12}}\bg(\bu(x,y_{j+\frac 12},t))dx.
\end{array}
\end{equation}
We use the Gauss quadrature to evaluate the above integrals to obtain numerical fluxes in order to guarantee the accuracy in space. For example, we evaluate $\bbf_{i+\frac 12,j}(t)$ for any time $t$, 
\begin{equation}
\dfr{1}{\De y_j}\int_{y_{j-\frac 12}}^{y_{j+\frac 12}} \bbf(\bu(x_{i+\frac 12},y,t))dy \approx \sum_{\ell=0}^k \om_\ell \bbf(\bu(x_{i+\frac 12},y_\ell,t)),
\end{equation}  
 where $y_\ell\in(y_{j-\frac 12},y_{j+\frac 12})$, $\ell=1,\cdots, k$, are  Gauss points and $\om_\ell$ are corresponding  weights. 
At each Gauss point $(x_{i+\frac 12},y_\ell,t_n)$,  we solve the quasi 1-D generalized Riemann problem (GRP) for \eqref{2D-law} , 
\begin{equation}
\begin{array}{l}
\bu(x,y_\ell,t_n) =\left\{
\begin{array}{ll}
\bu_{i,j}^n(x,y_\ell), \ \ \ &x<x_{i+\frac 12},\\
\bu_{i,j+1}^n(x,y_\ell), \ \ \ &x>x_{i+\frac 12}.\\
\end{array}
\right.
\end{array}
\label{2d-data}
\end{equation} 
In analogy with the 1-D case in \eqref{1d-grp}, we obtain
\begin{equation}
\bu_{i+\frac 12,j_\ell}^n: =\lim_{t\rw t_n+0} \bu(x_{i+\frac 12},y_\ell,t), \ \ \  \left(\frac{\pt \bu}{\pt t}\right)_{i+\frac 12,j_\ell}^n =\lim_{t\rw t_n+0} \frac{\pt \bu}{\pt t}(x_{i+\frac 12},y_\ell,t),
\end{equation}
and similarly for others.   The GRP solver for \eqref{2D-law} and \eqref{2d-data} is put in Appendix \eqref{app-B}.  Thus we propose the following two-stage algorithm for 2-D hyperbolic conservation laws \eqref{2D-law}. 
\vspace{0.2cm}

\n{\bf Algorithm  2-D.} 
\begin{enumerate}

\item[\bf Step 1.]  With the initial data $\bu^n(x,y)$ obtained by the HWENO interpolation, we compute the instantaneous values $\bu_{i_\ell, j+\frac 12}^n$, $\bu_{i+\frac 12, j_\ell}^n$, $(\pt\bu/\pt t)_{i_\ell,j+\frac 12}^n$and $(\pt\bu/\pt t)_{i+\frac 12,j_\ell}^n$ at every Gauss point.
\vspace{0.2cm} 

\item[\bf Step 2.]  Construct the intermediate values $\bu^*(x,y)$  at $t_*=t_n+\frac 12\De t$ using the formulae,
\begin{equation}
\begin{array}{l}
\bar\bu_{i,j}^* =\bar\bu_{i,j}^n -\dfr{\De t}{2\De x_i}[\bbf_{i+\frac 12,j}^*-\bbf_{i-\frac 12,j}^*]-\dfr{\De t}{2\De y_j}[\bg_{i,j+\frac 12}^*-\bg_{i,j-\frac 12}^*],  \\[3mm] 
\d \bbf_{i+\frac 12,j}^* = \sum_{\ell=0}^k\om_\ell \bbf(\bu(x_{i+\frac 12},y_\ell, t_n+\frac 14\De t)),  \ \ \   \bg_{i_\ell,j+\frac 12}^* = \sum_{\ell=0}^k\om_\ell \bg(\bu(x_\ell,y_{j+\frac 12}, t_n+\frac 14\De t)),\\[3mm]
 \d  \bu(x_{i+\frac 12},y_\ell, t_n+\frac 12\De t):= \bu_{i+\frac 12,j_\ell}^n +\frac{\De t}{2}\left(\frac{\pt\bu}{\pt t}\right)_{i+\frac 12,j_\ell}^n,\\[3mm]
 \d  \bu(x_\ell,y_{j+\frac 12}, t_n+\frac 12\De t):= \bu_{i_\ell,j+\frac 12}^n +\frac{\De t}{2}\left(\frac{\pt\bu}{\pt t}\right)_{i_\ell,j+\frac 12}^n.
\end{array}
\label{2d-al}
\end{equation} 
Then we use the HWENO interpolation to  reconstruct $\bu^*(x,y)$ and find the values $\bu_{i_\ell, j+\frac 12}^*$, $\bu_{i+\frac 12, j_\ell}^*$, $(\pt\bu/\pt t)_{i_\ell,j+\frac 12}^*$and $(\pt\bu/\pt t)_{i+\frac 12,j_\ell}^*$ at $t=t_n+\frac{\De t}2$ as done in Step 1. 
\vspace{0.2cm}

\item[\bf Step 3.]  Advance the solution to the next time level $t_n+\De t$, 
\begin{equation}
\begin{array}{l}
\bar \bu^{n+1}_{i,j} =\bar \bu_{i,j} ^n -\dfr{\De t}{\De x_j} [\bbf_{i+\frac 12,j}^{4th} -\bbf_{i-\frac 12,j}^{4th}]-\dfr{\De t}{\De y_j} [\bg_{i,j+\frac 12}^{4th} -\bg_{i,j-\frac 12}^{4th}],\\[3mm]
\d \bbf_{i+\frac 12,j}^{4th} =\sum_{\ell=0}^k \om_\ell\bbf_{i+\frac 12,j_\ell}^{4th},\ \ \ \bg_{i,j+\frac 12}^{4th} =\sum_{\ell=0}^k \om_\ell\bg_{i_\ell,j+\frac 12}^{4th}; \\[3mm]
\d \bbf_{i+\frac 12,j_\ell}^{4th}=\bbf(\bu_{i+\frac 12,j_\ell}^n) + \dfr{\De t}6\left[\dfr{\pt \bbf}{\pt t}(\bu_{i+\frac 12,j_\ell}^n)+ 2\dfr{\pt \bbf}{\pt t}(\bu_{i+\frac 12,j_\ell}^*)\right],\\[3mm] 
\d \bg_{i_\ell,j+\frac 12}^{4th}=\bg(\bu_{i_\ell,j+\frac 12}^n) +  \dfr{\De t}6\left[\dfr{\pt \bg}{\pt t}(\bu_{i_\ell,j+\frac 12}^n)+ 2\dfr{\pt \bg}{\pt t}(\bu_{i_\ell,j+\frac 12}^*)\right].\\[3mm] 
\end{array}
\label{2d-al2}
\end{equation} 

\end{enumerate}

\vspace{0.2cm}

\section{Numerical Examples} 

In this section we provide several examples to validate the performance of the proposed approach.  The examples include linear and nonlinear scalar conservation laws,  1-D Euler equations  and 2-D Euler equations. The order of   accuracy will be tested.  All results are obtained with CFL number $0.5$ except the large density ratio problem for which the CFL number is taken to be $0.2$.  We use GRP4-HWENO5 to denote the algorithm with  the GRP solver and the HWENO fifth order accurate spatial  reconstruction, and use RK4-WENO5 to denote the algorithm with the WENO fifth order accurate spatial reconstruction and fourth order accurate temporal R-K iteration. 

 \subsection{Scalar conservation laws} 
 
 We  use our approach to solve two examples of scalar conservation laws. 
 \vspace{0.2cm}
 
 \n {\bf Example 1. } The first example is a linear advection equation with a periodic boundary condition,
 \begin{equation}
 u_t + u_x=0, \ \ \ \ \ \ \  \ u(x,0)=\sin(\pi x). 
 \end{equation} 
 The solution is computed over the space interval $[0, 2]$ and the results are displayed in Table 1. We can see that the accuracy is achieved as expected. 

 
\begin{table*}[!htbp]
  \centering
  \begin{tabular}{|l|r|r|r|r|r|r|r|r|}
    \hline
$m$&      \multicolumn{4}{|c|}{RK4-WENO5}     &     \multicolumn{4}{|c|}{GRP4-HWENO5}     \\\cline{2-9}
   &$ L_1 $ error&order&$ L_\infty $ error&order&$ L_1 $ error&order&$ L_\infty $ error &order\\\hline
40 & 4.47(-4)    &4.91&   3.81(-4)      & 4.73 & 1.67(-4)   & 5.07&   1.60(-4)      & 4.91\\
80 & 1.40(-5)    &5.00&   1.27(-5)      & 4.91 & 5.28(-6)   & 4.99&   5.10(-6)      & 4.97\\
160& 4.37(-7)    &5.00&   3.97(-7)      & 5.00 & 1.79(-7)   & 4.88&   1.60(-7)      & 4.99\\
320& 1.37(-8)    &5.00&   1.25(-8)      & 4.99 & 7.19(-9)   & 4.64&   5.68(-9)      & 4.82\\
640& 4.30(-10)   &4.99&   3.77(-10)     & 5.05 & 3.60(-10)  & 4.32&   2.86(-10)     & 4.31\\\hline
  \end{tabular}
  \caption[small]{The comparison of $L_1$, $L_\iy$ errors and convergence order for a convection equation.  The schemes are RK4-WENO5 and GRP4-HWENO5 with $ m $ cells. The results are shown at  time $ t=10 $.}
  \label{tab:scalar_smooth}
\end{table*}
 
 \vspace{0.2cm} 
 
 \n{\bf Example 2.}  The second example is taken for the Burgers equation with a periodic boundary condition \cite{DG-2}, 
\begin{equation}
u_t +\left(\frac{u^2}{2}\right)_x=0, \ \ \ u(x,0) = \frac 14 + \frac 12\sin(\pi x). 
\end{equation} 
The solution is smooth up to the time $t=2/\pi$ and develops a shock that moves to interact with a rarefaction, as shown in Figure \ref{Fig-scalar_burgers} . The errors and convergence  order are shown in Table 2. 


\begin{figure}[!htb]
\centering
 \includegraphics[width=\textwidth]{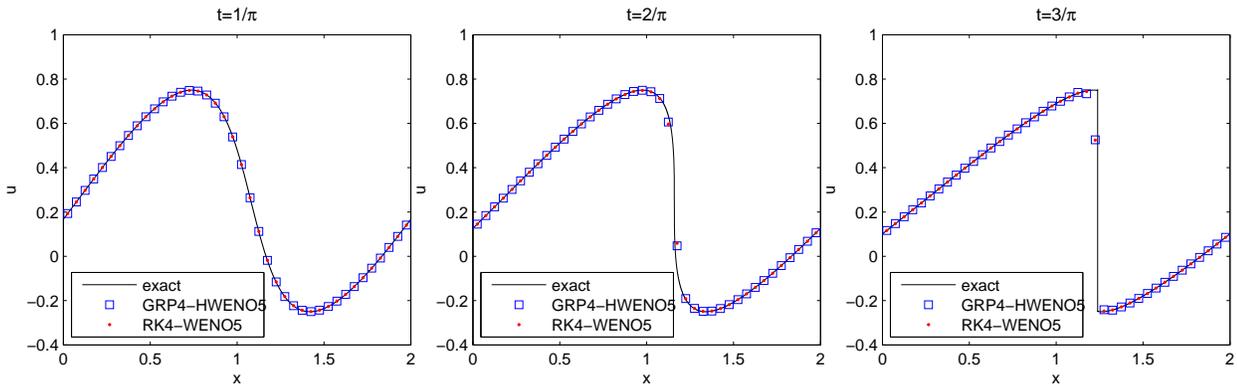}
 \caption[small]{Numerical solutions of the Burgers equation. The schemes are RK4-WENO5 and GRP4-HWENO5 which are implemented with $40$ cells. The results are shown at $ t = 1/\pi $ (left), $ t = 2/\pi $ (middle), $ t = 3/\pi $ (right), respectively. }
 \label{Fig-scalar_burgers}
 \end{figure} 
\vspace{0.2cm}
\begin{table*}[!htbp]
  \centering
  \begin{tabular}{|l|r|r|r|r|r|r|r|r|}
    \hline
$m$&      \multicolumn{4}{|c|}{RK4-WENO5}     &     \multicolumn{4}{|c|}{GRP4-HWENO5}     \\\cline{2-9}
   &$ L_1 $ error&order&$ L_\infty $ error&order&$ L_1 $ error&order&$ L_\infty $ error &order\\\hline
40 & 3.60(-5)    &3.79&   1.74(-4)      & 2.94 & 9.07(-6)   & 4.39&   4.32(-5)      & 3.58\\
80 & 1.84(-6)    &4.29&   8.17(-6)      & 4.41 & 5.53(-7)   & 4.03&   3.01(-6)      & 3.84\\
160& 7.06(-8)    &4.71&   4.31(-7)      & 4.25 & 2.48(-8)   & 4.48&   1.83(-7)      & 4.05\\
320& 1.84(-9)    &5.26&   9.53(-9)      & 5.50 & 8.91(-10)  & 4.80&   2.97(-9)      & 5.94\\
640& 4.63(-11)   &5.31&   2.94(-10)     & 5.02 & 4.93(-11)  & 4.17&   1.96(-10)     & 3.91\\\hline
  \end{tabular}
  \vspace{0.2cm}
  
  \caption[small]{The comparison of $L_1$, $L_\iy$ errors and convergence order for the Burgers equation. The schemes are RK4-WENO5 and GRP4-HWENO5 with $ m $ cells. The results are shown at time $ t=1/\pi $.}
  \label{tab:scalar_burgers}
\end{table*}
 
 \subsection{One-dimensional Euler equations}  We provide  several examples for 1-D compressible Euler equations, 
 \begin{equation}
 \bu =(\rho, \rho v, \rho E)^\top, \ \ \ \ \ \bbf(\bu) =(\rho v, \rho v^2 +p, v(\rho E+p))^\top,
 \label{Euler}
 \end{equation} 
 where  $\rho$ is the  density, $v$ is the velocity, $p$ is the pressure and $E= v^2/2+ e$ is the total energy, $e=\frac{p}{(\gm-1)\rho}$ is the internal energy  for polytropic gases.  We test several standard examples to validate the proposed scheme. 
 \vspace{0.2cm} 
 
\n{\bf Example 3. Smooth problem.} 
In order to verify the numerical accuracy of the present fourth order accurate scheme, we check the numerical results for a  smooth problem whose initial data is
\begin{equation}\label{eq:smooth}
     \rho(x,0)   =     1+0.2\text{sin}(x), \ \ \ v(x,0)=1, \ \  p(x,0)=1.
\end{equation}
The periodic boundary conditions are used again. The results are shown in Table 3, which verifies the  expected accuracy order.  

\begin{table*}[!htbp]
  \centering
  \begin{tabular}{|l|r|r|r|r|r|r|r|r|}
    \hline
$m$&     \multicolumn{4}{|c|}{RK4-WENO5}      &       \multicolumn{4}{|c|}{GRP4-HWENO5}    \\\cline{2-9}
   &$ L_1 $ error&order&$ L_\infty $ error&order&$ L_1 $ error &order&$ L_\infty $ error &order\\\hline
40 &  8.92(-5)   & 4.91 &   7.64(-5)    & 4.72 &   3.33(-5)  & 5.07&   3.13(-5)      & 4.90 \\
80 &  2.78(-6)   & 5.00 &   2.53(-6)    & 4.91 &   1.04(-6)  & 5.01&   9.86(-7)      & 4.97 \\
160&  8.61(-8)   & 5.01 &   7.83(-8)    & 5.02 &   3.31(-8)  & 4.97&   3.05(-8)      & 5.01 \\
320&  2.59(-9)   & 5.06 &   2.23(-9)    & 5.13 &   1.12(-9)  & 4.88&   1.01(-9)      & 4.91 \\
640&  7.07(-11)  & 5.19 &   6.15(-11)   & 5.18 &   4.37(-11) & 4.68&   4.19(-11)     & 4.60 \\\hline
  \end{tabular}
  \vspace{0.2cm} 
  
  \caption[small]{The comparison of $L_1$, $L_\iy$ errors and convergence order for the Euler equations in Example 3. The schemes are RK4-WENO5 and GRP4-HWENO5 with $ m $ cells. The results are shown at  time $ t=10 $.}
  \label{tab:Euler_smooth}
\end{table*}



 \n{\bf Example 4.  Shock-turbulence interaction problem.}   This example was proposed in \cite{Shu-Osher} to model shock-turbulence interactions. The initial data is take as
 \begin{equation}
(\rho,v,p)(x,0)=\left\{ \begin{array}{ll}
(3.857143,2.629369, 10.333333), \ \ \ \ &\mbox{for } x<-4,  \\[3mm] 
  (1 + 0.2 \sin(5x), 0, 1),&\mbox{for } x\geq -4. 
 \end{array}
 \right.
  \end{equation} 
  The result is shown in Figure  \ref{Fig-Shu-Osher} and it is comparable with those by other schemes. 

  \vspace{0.2cm} 
  
\begin{figure}[!htb]
 \centering
 \includegraphics[width=\textwidth]{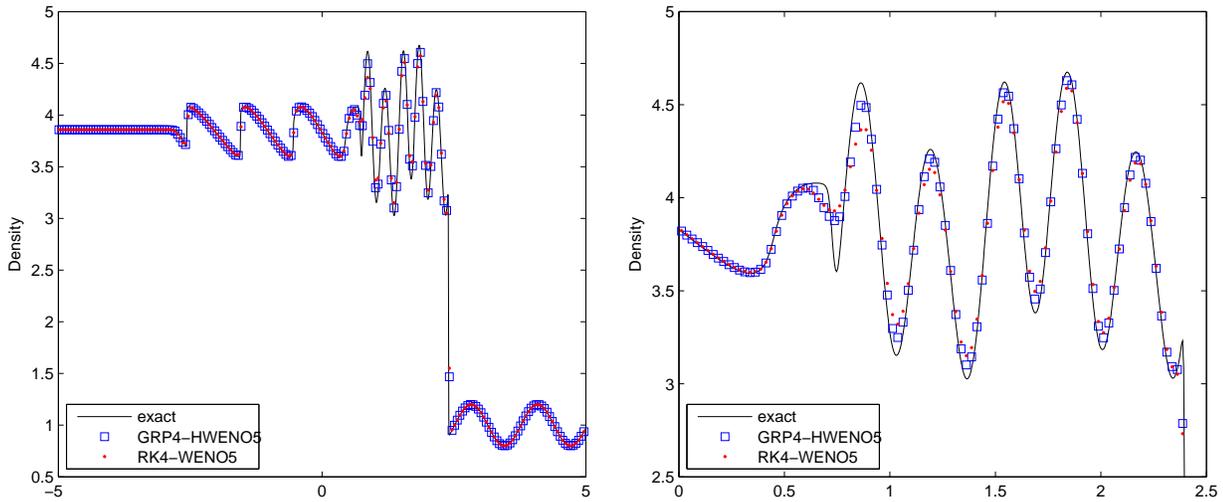}
 \label{Fig-Shu-Osher}
  \caption[small]{The comparison of the density profile for the shock-turbulence interaction problem. The schemes used are RK4-WENO5  and GRP4-HWENO5 with $400$ cells. The solid lines are the reference solution.}

\end{figure}
\vspace{0.2cm}

 \n{\bf Example 5.  Woodward-Colella problem.}  This is the Woodward-Colella interacting blast wave problem. 
The gas is at rest and ideal with $\gm=1.4$, and the density is
everywhere unit. The pressure is $p=1000$ for $0\leq x <0.1$ and
$p=100$ for $0.9<x\leq 1.0$, while it is only $p=0.01$ in
$0.1<x<0.9$. Reflecting boundary conditions are applied at both
ends. Both the GRP4-HWENO5 scheme  and the RK4-WENO5 scheme  could give a well-resolved solution using $800$ grids.  See   Figure \ref{fig:Wood}. The reference solution is computed with $4000$ grids.  

\begin{figure}[!htbp]
\centering
\includegraphics[width=\textwidth]{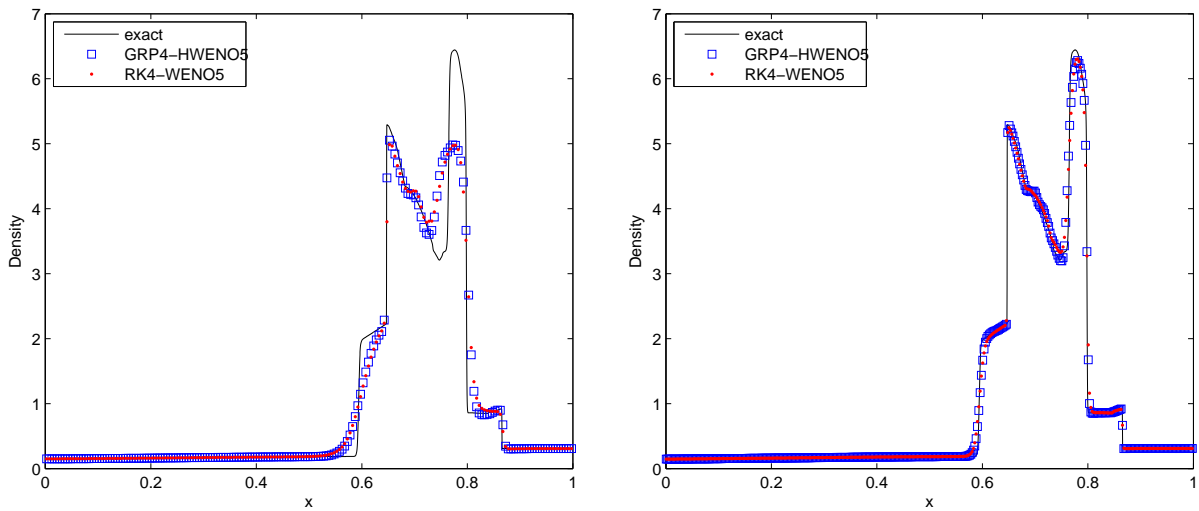}
  \caption[small]{The comparison of the density profile for the Woodward-Colella problem. The schemes are RK4-WENO5 and GRP4-HWENO5 with 200 cells (left) and 800 cells (right, 400 are shown), respectively. The solid lines are the reference solution.}
\label{fig:Wood}
\end{figure}
\vspace{0.2cm} 

 \n{\bf Example 6. Large pressure ratio problem. }  The large pressure ratio problem is first presented in \cite{Tang-Liu}  to test the ability to capture extremely strong rarefaction waves and its influence on the shock
location. In the original paper \cite{Tang-Liu}, it shows that most MUSCL--type schemes have defects in resolving, even with very fine mesh, the correct wave
structures.
In this problem, initially the pressure and density ratio between two neighboring states are very high.
The initial data is  $(\rho,v,p)=(10000,0,10000)$ for $0\leq x<0.3$ and $(\rho,v,p)=(1,0,1)$ for $0.3\leq x \leq 1.0$. 
The results with  $400$ points  are shown in Figure \ref{fig:ratio400}, by  GRP4-HWENO5 and RK4-WENO5, respectively.
With $400$ grid points, the GRP4-HWENO scheme gives  perfect results,  while the RK4-WENO5 fails to achieve  that.  


\begin{figure}[htp]
\centering
\includegraphics[width=\textwidth]{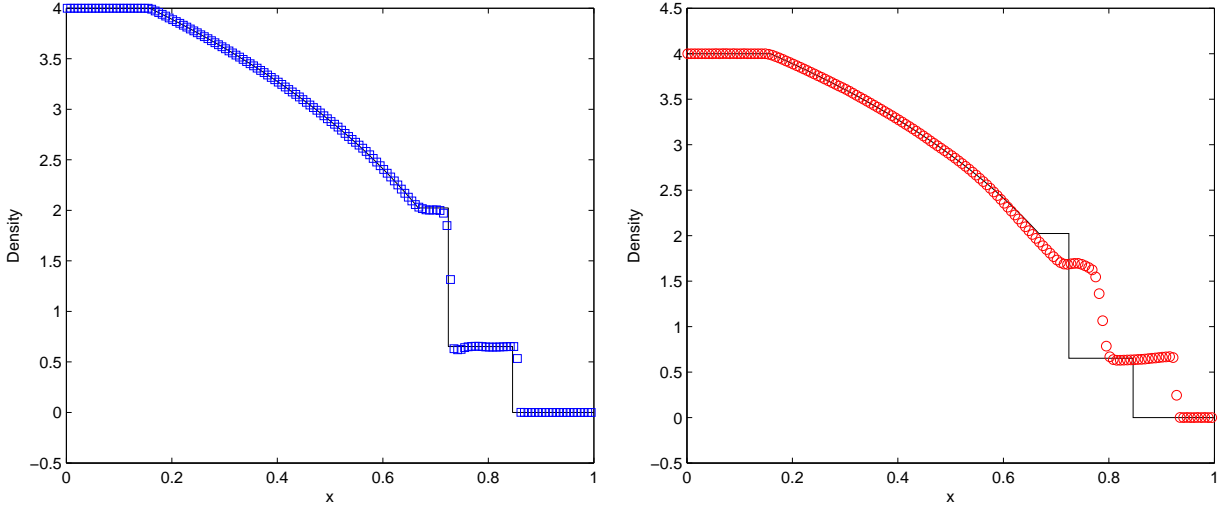}
  \caption[small]{The comparison of the density for the large pressure ratio problem. The magnitude is scaled by $10^4$. The schemes are GRP4-HWENO5 (left) and RK4-WENO5 (right) with $300$ cells. The solid lines are the reference solution.}
\label{fig:ratio400}
\end{figure}
\vspace{0.2cm}

\subsection{2-D Examples.}  We provide several two-dimensional examples to validate the proposed approach.  The governing equations are the 2-D Euler equations,
\begin{equation}
\begin{array}{l}
\bu=(\rho,\rho u,\rho v,\rho E)^\top, \\[3mm] 
\bbf(\bu) =(\rho u, \rho u^2+p,\rho uv, u(\rho E+p))^\top,\\[3mm]
\bg(\bu) =(\rho v, \rho uv, \rho v^2+p, v(\rho E+p))^\top,
\end{array}
\label{2d-Euler}
\end{equation} 
where $(u,v)$ is the velocity, $E=\frac{u^2+v^2}2+e$, $e=\frac{p}{(\gm-1)\rho}$.  The first example is about the isentropic vortex problem to test the accuracy.   The other examples aim to verify the expected performance of this approach. 

\vspace{0.2cm}

\n{\bf Example 7. Isentropic vortex problem.} In this first 2-D isentropic vortex example we  want to verify the numerical accuracy of our scheme.  Initially  the mean flow is given with $ \rho = 1 $, $ p = 1 $, and $ (u,v)=(1,1) $. Then an isentropic vortex is put on this mean flow 
\begin{equation*}
\begin{array}{c}
  (\delta u, \delta v) = \frac{\epsilon}{2\pi}e^{0.5(1-r^2)}(-\bar{y}, \bar{x}), \\[10pt]
  \delta T = - \frac{(\gamma - 1) \epsilon^2}{8\gamma\pi^2}e^{1-r^2}, ~ \delta S=0,
\end{array}
\end{equation*}
where $ (\bar{x},\bar{y}) = (x-5,y-5)$, $ r^2 = \bar{x}^2+\bar{y}^2 $, and the vortex strength is often set to be $ \epsilon = 5.0 $. The computation is performed in the domain $ [0,10]\times[0,10] $, extended periodically in both directions. The accuracy is achieved with the expected  order $4$. 
\vspace{0.2cm} 

\begin{table*}[!htbp]
  \centering
  \begin{tabular}{|l|r|r|r|r|r|r|r|r|}
    \hline
$m$&     \multicolumn{4}{|c|}{RK4-WENO5}      &       \multicolumn{4}{|c|}{GRP4-HWENO5}    \\\cline{2-9}
   &$ L_1 $ error&order&$ L_\infty $ error&order&$ L_1 $ error &order&$ L_\infty $ error &order\\\hline
40 &  1.79(-4)   & 3.80 &   3.29(-3)    & 3.76 &   7.70(-3)  & 3.87 &   1.92(-3)  & 3.71 \\
80 &  6.92(-6)   & 4.69 &   1.96(-4)    & 4.07 &   3.47(-4)  & 4.47 &   1.26(-4)  & 3.93 \\
160&  2.03(-7)   & 5.09 &   4.95(-6)    & 5.31 &   1.05(-5)  & 5.05 &   3.26(-6)  & 5.28 \\
320&  7.83(-9)   & 4.70 &   1.96(-7)    & 4.66 &   3.47(-7)  & 4.92 &   1.54(-7)  & 4.40 \\
640&    -        &   -  &        -      &  -   &   1.03(-8)  & 5.08 &   5.33(-9)  & 4.86 \\
\hline
  \end{tabular}
  \vspace{0.2cm}
  
  \caption[small]{The comparison of $L_1$, $L_\iy$ errors and convergence order for the isentropic vortex problem of the Euler equations. The schemes are RK4-WENO5 and GRP4-HWENO5 with $ m \times m $ cells. The results are given at  time $ t=2 $.}
  \label{tab:Isen_vortex}
\end{table*}

\n{\bf Example 8.  Two-dimensional Riemann problems.} We provide three examples for two-dimensional Riemann problems, as shown in Figure \ref{fig:Riemann4in1_A4H5}.   These examples are taken from \cite{Han} and  involve the interactions of shocks, the interaction of shocks with vortex sheets and the interaction of vortices. Here we use $S$ represents a shock,  $J$ a vortex sheet and $R$ a rarefaction wave.  The computation is implemented over the domain $[0,1]\times [0,1]$.  The output time is specified below case by case.

\vspace{0.2cm}

\n {\bf a.  Interaction of shocks and vortex sheets } $S_{12}^{+}J_{23}^{-}J_{34}^{+}S_{41}^{-}$.
   The initial data are 
\begin{equation}
(\rho,u,v,p)(x,y,0) =\left\{
\begin{array}{ll}
(1.4,8,20,8), \ & 0.5<x<1.0,0.5<y<1.0,\\
(-4.125,4.125,-4.125,-4.125), & 0< x<0.5,0.5<y<1.0,\\
(-4.125,-4.125,-4.125,4.125), & 0<x<0.5,0<y<0.5,\\
(1,116.5,116.5,116.5), & 0.5<x<1.0,0<y<0.5.
\end{array}
\right.
\end{equation}
The output time is $0.26$. 

 \vspace{0.2cm} 
 
 \n {\bf b.  Interaction of shocks, rarefaction waves and vortex sheets}  $J_{12}^{+}S_{23}^{-}J_{34}^{-}R_{41}^{+}$.
   The initial data are 
\begin{equation}
(\rho,u,v,p)(x,y,0) =\left\{
\begin{array}{ll}
(1, 2, 1.0625, 0.5179), \ & 0.5<x<1.0,0.5<y<1.0,\\
(0, 0, 0, 0), & 0< x<0.5,0.5<y<1.0,\\
(0.3, -0.3, 0.2145, -0.4259), & 0<x<0.5,0<y<0.5,\\
(1, 1, 0.4, 0.4), & 0.5<x<1.0,0<y<0.5.
\end{array}
\right.
\end{equation}
The output time is $ t=0.055 $.

 \vspace{0.2cm} 
 
 \n { \bf c.  Interaction of rarefaction waves and vortex sheets}  $R_{12}^{+}J_{23}^{+}J_{34}^{-}R_{41}^{-}$.
   The initial data are 
\begin{equation}
(\rho,u,v,p)(x,y,0) =\left\{
\begin{array}{ll}
(1,0.5197,0.8,0.5197), \ & 0.5<x<1.0,0.5<y<1.0,\\
(0.1,-0.6259,0.1,0.1), & 0< x<0.5,0.5<y<1.0,\\
(0.1,0.1,0.1,-0.6259), & 0<x<0.5,0<y<0.5,\\
(1,0.4,0.4,0.4), & 0.5<x<1.0,0<y<0.5.
\end{array}
\right.
\end{equation}
The output time is $0.3$. 

From the results we can see that this scheme can capture very small scaled vortices resulting from the interaction of vortex sheets. The resolution of vortices is  comparable to that by the adaptive moving mesh GRP method (cf. \cite{Han}). 

 \vspace{0.2cm}

\begin{figure}[htp]
\centering
\includegraphics[width=\textwidth]{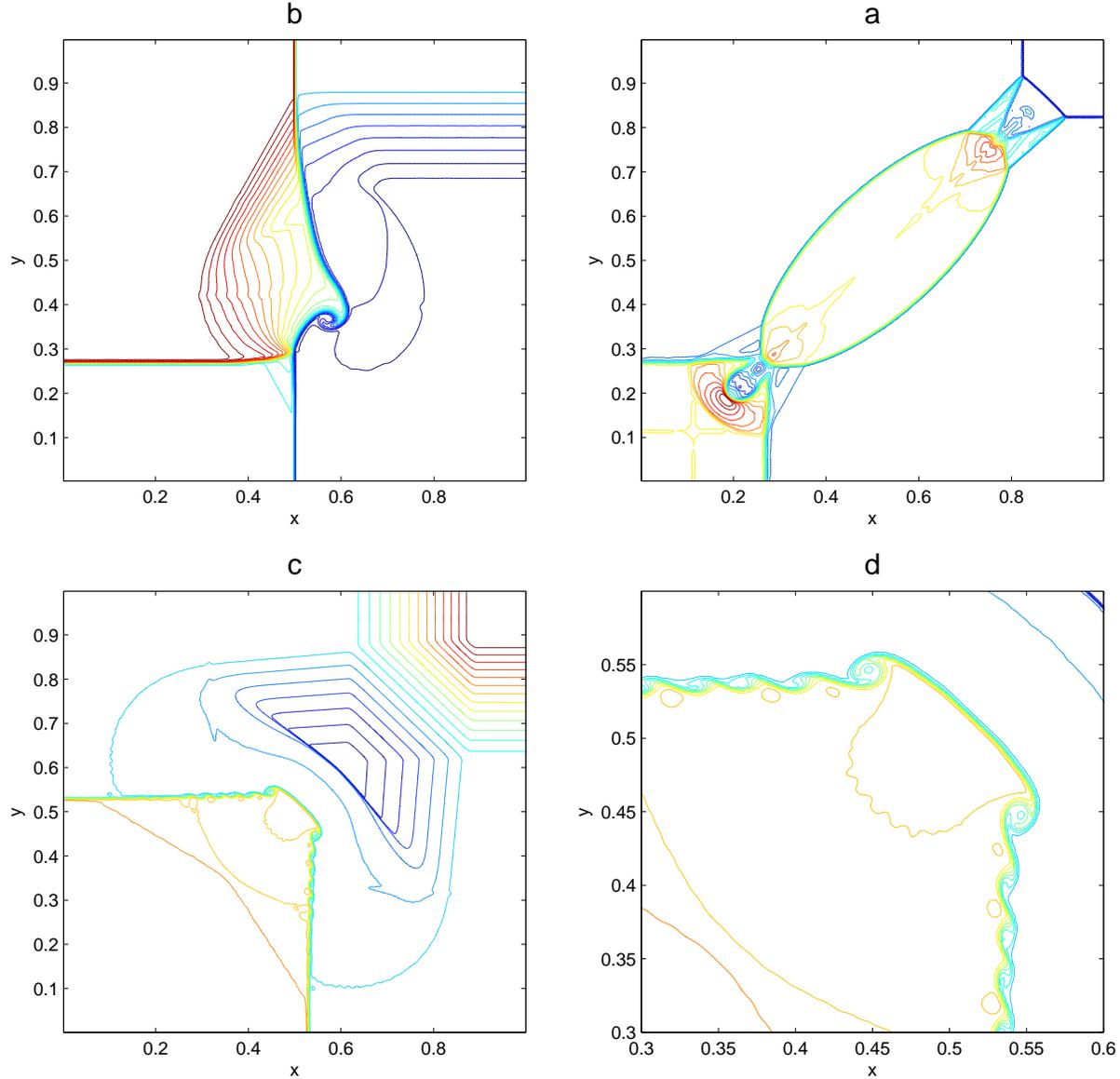}
\caption{The density contours of three 2-D Riemann problems computed with GRP4-HWENO5.  a. $ [J_{12}^{+}S_{23}^{-}J_{34}^{-}R_{41}^{+}] $ with $ 200 \times 200 $ cells. b. $ [S_{12}^{+}J_{23}^{-}J_{34}^{+}S_{41}^{-}] $ with $ 300 \times 300 $ cells. c. $ [R_{12}^{+}J_{23}^{+}J_{34}^{-}R_{41}^{-}] $ with $ 500 \times 500 $ cells. d. Local enlargement of c.}
\label{fig:Riemann4in1_A4H5}
\end{figure}
  
\n{\bf Example 9. The double mach reflection problem.} This is again a standard test problem to display the performance of high resolution schemes. The computational domain for this problem is $ [0,4] \times [0,1] $, and $ [0,3] \times [0,1]$ is shown. The reflection wall lies at the bottom of the computaional domain starting from $ x=\frac{1}{6} $. Initially a right-moving Mach 10 shock is positioned at $ x=\frac{1}{6} ,y=0$ and makes a $ \frac{\pi}{3} $ angle with the $x$-axis. The results are shown in Figures \ref{fig:DM240_A4H5} and \ref{fig:DMGRP-H480} with excellent performance.

\begin{figure}[htp]
\centering
\includegraphics[width=\textwidth]{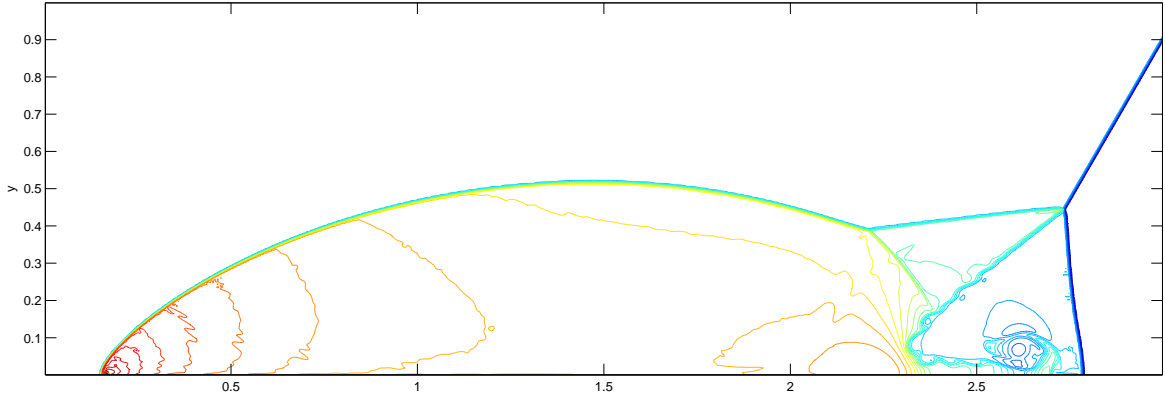}
\caption{The contours of the density for the double mach reflection problem. GRP4-HWENO5 is implemented with $960\times 240$ cells and  the result is shown at $ t = 0.2 $. 30 contours are drawn.}\label{fig:DM240_A4H5}
\end{figure}
\begin{figure}[htp]
\centering
\includegraphics[width=\textwidth]{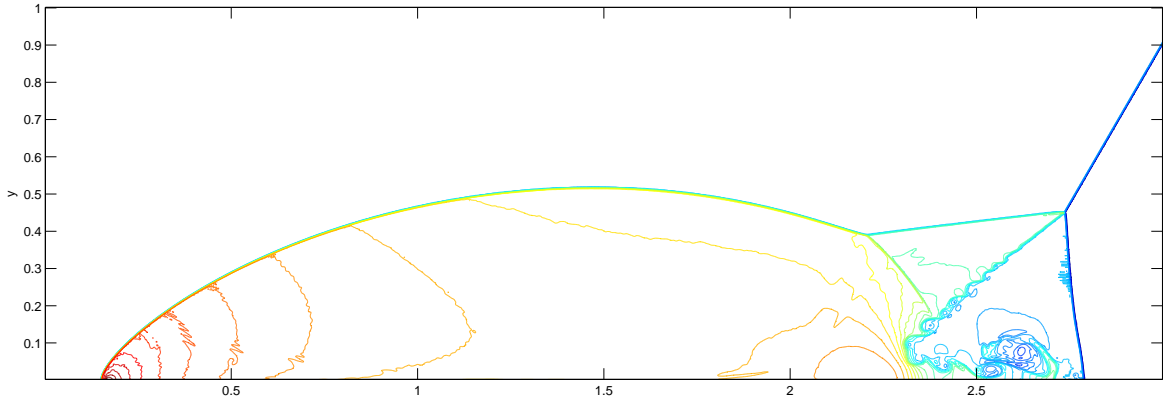}
\caption{The contours of the density of the double mach reflection problem. GRP2-HWENO5 is implemented with $1920\times 480$ cells and  the result is shown at $ t = 0.2 $. 30 contours are drawn.}\label{fig:DMGRP-H480}
\end{figure}


\vspace{0.2cm}

\n{\bf Example 10. The shock vortex interaction problem} This example describes the interaction between a stationary shock and a vortex, the computational domain is taken to be $ [0,2] \times [0,1] $. A stationary Mach 1.1 shock is positoned at $ {x = 0.5} $ and normal to the $ x $-axis. Its left state is $ (\rho, u, v, p) = (1, \sqrt{\gamma}M, 0, 1)$, where $ M $ is the mach number of the shock. A small vortex is superposed to the flow left to the shock and centers at $ (x_c, y_c) = (0.25, 0.5)$. The vortex can be considered as a perturbation to the mean flow. The perturbations to the velocity ($u$, $v$), the temperature ($ T=\frac{p}{\rho} $) and the entropy ($ S=\frac{p}{\rho^{\gamma}} $) are:
\begin{equation*}
  \begin{array}{c}
  (\delta u, \delta v) = \frac{\epsilon}{r_c} e^{\alpha(1-\tau^2)}(\bar{y}, -\bar{x}), \\[10pt]
  \delta T = - \frac{(\gamma - 1) \epsilon^2}{4\alpha\gamma}e^{2\alpha(1-\tau^2)}, ~ \delta S=0.
  \end{array}
\end{equation*}
In our case, we set $ \epsilon=0.3$ and $ \alpha=0.204 $. The computation is performed on a $ 400 \times 100 $ uniform mesh. The results (the pressure contours) are shown in Figure \ref{fig:SVI_A4H5}.
\begin{figure}[htp]
\centering
\includegraphics[width=\textwidth]{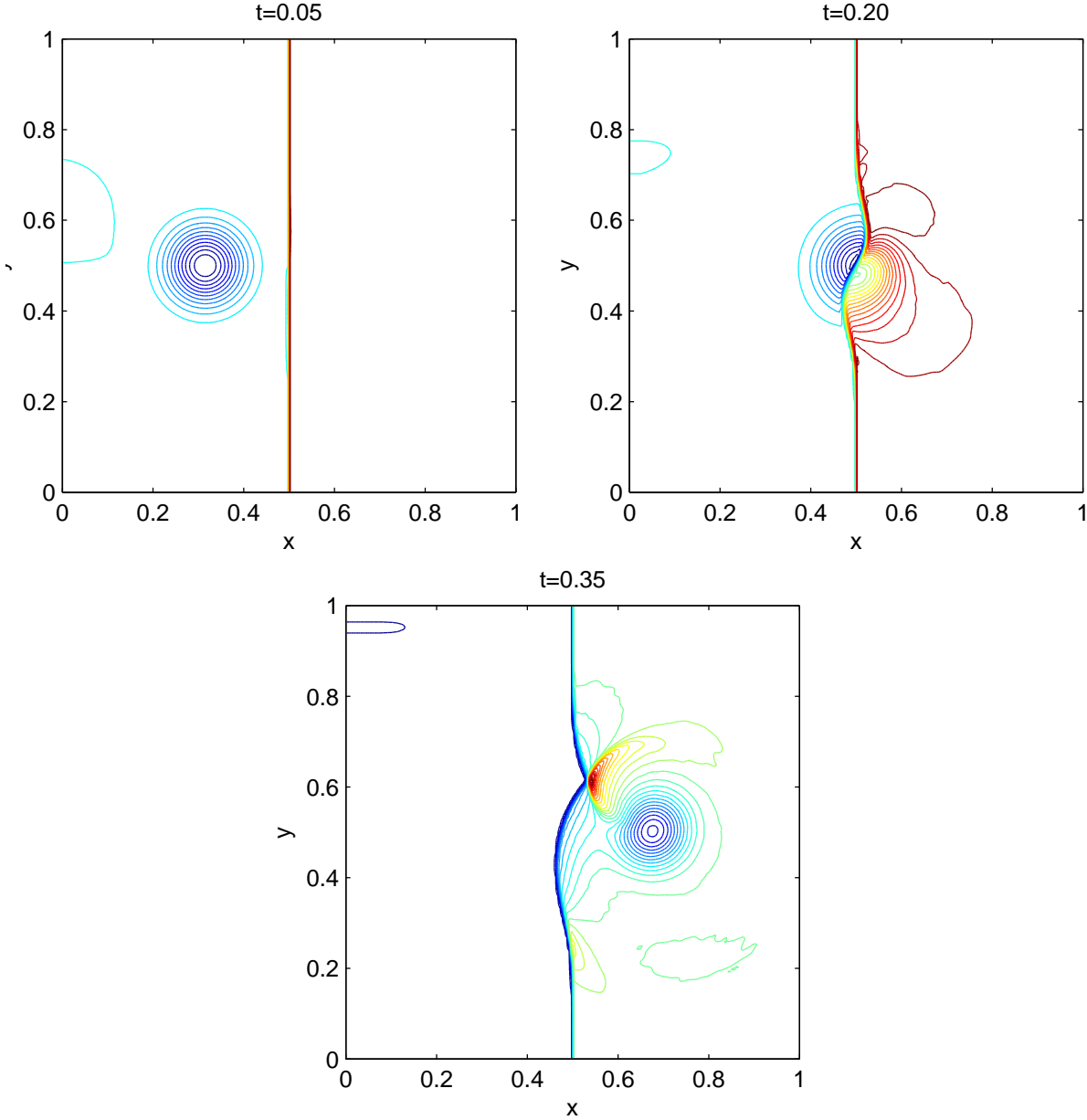}
\caption{The contours of the pressure for the shock vortex interaction problem. GRP4-HWENO5 with $400\times 200$ cells is implemented and $30$ contours are drawn.}\label{fig:SVI_A4H5}
\end{figure}

\section{Discussions and Prospectives}

This paper proposes a two-stage fourth order accurate temporal discretization for time-dependent problems based on the L-W type flow solvers.  The particular applications are given for hyperbolic conservation laws.  Based on HWENO interpolation technology \cite{Qiu},   a scheme with  a fifth order accuracy in space and a fourth order accuracy in time is developed.  A number of numerical examples are provided to validate the accuracy of the scheme and  its computational performance   for complex flow problems. 
\vspace{0.2cm}

The current temporal  discretization is different from  the classical R-K approach. As discussed in Sections 2 and 3,   the present approach is  of the Hermite type while the R-K approach is of Simpson type.   The L-W approach  with coupled space and time evolution is the basis for the development of the current high order method.  Our approach can be viewed as the extension of the L-W method from second order to even higher order accuracy,   without using successive replacement of temporal derivatives by spatial derivatives in the one-stage method.  This technique is particularly useful for nonlinear systems. 
\vspace{0.2cm}

In this paper we just apply this approach to hyperbolic conservation laws in the finite volume framework over rectangular meshes.  However,   this approach can be applied to any time-dependent problems as long as  L-W  type solvers are available over any type of computational mesh.  In the future studies, we will extend this approach to other formulations, e.g., DG formulation \cite{Li-Wang}, to other systems  e.g.,  the Navier-Stokes equations \cite{Du-Li}.  
\vspace{0.2cm}

This work is just a starting point for the design of high order accurate methods  and a lot of theoretical problems remain for the further study,   such as numerical stability.  Nevertheless, the numerical experiments clearly show that the current fourth order scheme can use a  CFL number as large as that for  the second order GRP scheme. Indeed, the CFL number can be taken even larger than $1/2$ if the waves  computed are not very strong. The large time step, in comparison with other high order schemes, does not  decrease the accuracy of the scheme. So this approach will be efficient  for the simulation of turbulence flows with multi-scale structures by taking a large time step and the coupling of the spatial and temporal numerical flow evolution.

\appendix
\ 

\section{The GRP solver}  \label{app-B}

This appendix includes the GRP solver  used in the coding process just for completeness and readers' convenience. The details can be found  in \cite{Li-1} for the Euler equations and \cite{Li-2} for general hyperbolic systems,
\begin{equation}
\bu_t+\bbf(\bu)_x=\bg(\bu,x),
\label{1d-balance}
\end{equation} 
where $\bg(\bu,x)$ is a source term. This paper  focuses on the homogeneous case, $\bg(\bu,x)\ev0$ for 1-D case. As far as 2-D case is concerned with,  the effect tangential to cell interfaces can be regarded as a source and therefore the 2-D  GRP solver can be derived using the similar idea to that for 1-D GRP solver.    
\vspace{0.2cm} 

\subsection{1-D GRP}

The 1-D GRP solver assumes that the initial data consist of two pieces of polynomials,
\begin{equation}
\bu(x,0) =\left\{
\begin{array}{ll}
\bu_-(x), \ \ \ & x<0,\\[3mm]
\bu_+(x), & x>0, 
\end{array}
\right.
\label{GRP-data}
\end{equation}
where  $\bu_\pm(x)$ are two polynomials with  limiting states 
\begin{equation}
\begin{array}{c}
\d\bu_\ell =\lim_{x\rw 0-0} \bu_-(x),  \ \ \ \ \ \bu_r=\lim_{x\rw 0+0} \bu_+(x); \\[3mm]
\d \bu_\ell'=\lim_{x\rw 0-0} \bu_-'(x),  \ \ \ \ \ \bu_r'=\lim_{x\rw 0+0} \bu'_+(x). 
\end{array}
\end{equation} 
In the present study, we use the HWENO method  in \cite{Qiu} to construct the initial data and therefore $\bu_\pm(x)$
are two pieces of polynomials of order five. 
\vspace{0.2cm} 

The GRP solver has two versions: (i) Acoustic version; (ii) Genuinely nonlinear version. 

\subsubsection{Acoustic GRP solver}  The acoustic GRP deals with weak discontinuities or smooth flows and assumes that 
\begin{equation}
\|\bu_\ell-\bu_r\|\ll 1. 
\end{equation}
However, we emphasize that the difference $\bu_\ell'-\bu_r' $ is not necessarily small.  Then we denote by 
\begin{equation}
\bu_0\approx\bu_\ell\approx \bu_r, 
\end{equation}
and linearize \eqref{1d-balance} around  $\bu_0$ as
\begin{equation}
\bu_t+ A(\bu_0) \bu_x=0,  \ \ \ \ \ A(\bu_0) := \dfr{\pt \bbf(\bu_0)}{\pt \bu}.  
\end{equation}  
Then the instantaneous time derivative of $\bu$ is computed as,
\begin{equation}
\left(\dfr{\pt \bu}{\pt t}\right)_0 := \lim_{t\rw 0+0}  \dfr{\pt \bu}{\pt t}(0,t) = -[R\La^+ R^{-1} \bu_\ell' + R\La^- R^{-1} \bu_r'],
\end{equation} 
where $\La=\mbox{diag}(\la_1, \cdots, \la_m)$, $\la_i$, $i=1,\cdots, m$ are the eigenvalues of $A(\bu_0)$, $R$ is the (left) eigenmatrix of $A(\bu_0)$, $\La^+ =\mbox{diag}(\max(\la_i,0))$, $\La^- =\mbox{diag}(\min(\la_i,0))$. \vspace{0.2cm}

The acoustic GRP is named as the $G_1$ scheme  in the series of GRP papers and it is consistent with the ADER solver by Toro \cite{Toro-ADER}. 
\vspace{0.2cm}

\subsubsection{Nonlinear GRP solver.} As  the jump at $x=0$ is large, the acoustic GRP is not sufficient to resolve the resulting strong discontinuities. Any ``rough" approximation is dangerous  since  the error is measured with jump $\|\bu_r-\bu_r\|$, which is not proportional to the mesh size in the practical computation and may lead to  large numerical discrepancy.  Therefore, we have to analytically solve the associated generalized Riemann problem \eqref{1d-balance}-\eqref{GRP-data} and derive 
 the ``genuinely" nonlinear GRP solver, which is named as the $G_\iy$ GRP.  This version is interpreted as the L-W approach plus the tracking of strong discontinuities. 
 \vspace{0.2cm} 
 
 Here we include the resolution of GRP \eqref{1d-balance}-\eqref{GRP-data} for the Euler equations \eqref{Euler}. The instantaneous value $\bu_0$ is obtained by the Riemann solver and $(\pt\bu/\pt t)_0$ is obtained by solving a pair of algebraic equations essentially, 
 \begin{equation}
 \begin{array}{ll}
 a_\ell\left(\dfr{\pt v}{\pt t}\right)_0 + b_\ell \left(\dfr{\pt p}{\pt t}\right)_0=d_\ell,\\[3mm]
 a_r\left(\dfr{\pt v}{\pt t}\right)_0 + b_r\left(\dfr{\pt p}{\pt t}\right)_0=d_r,
 \end{array}
 \end{equation} 
 where the coefficients  $a_i, b_i, d_i$, $i=1,2$, are given explicitly in terms of the initial data \eqref{GRP-data}, and their formulae can be found in \cite{Li-1}. 
 \vspace{0.2cm}


 Since the variation of entropy $s$ is precisely quantified, the instantaneous time derivative of the density is then obtained using the equation of state $p=p(\rho, s)$, 
\begin{equation}
dp =c^2 d\rho +\dfr{\pt p}{\pt s} ds. 
\end{equation}
\vspace{0.2cm}
 
\subsection{Quasi-1-D GRP solver}  As the two-dimensional case are dealt with,  we need to solve a  so-called quasi-1-D GRP
\begin{equation}
\begin{array}{l}
\bu_t +\bbf(\bu)_x +\bg(\bu)_y=0, \\[3mm]
\bu(x,y,0)=\left\{
\begin{array}{ll}
\bu_\ell(x,y), \ \ \ \ & x<0, \\
\bu_r(x,y) & x>0,
\end{array}
\right.
\end{array}
\label{Q1-GRP} 
\end{equation}
where $\bu_\ell(x,y)$ and $\bu_r(x,y)$ are two polynomials defined on the two neighboring computational cells, respectively.  Since we just want to construct the  flux  normal to cell interfaces, the tangential effect can be regarded as 
a source. Therefore, we rewrite \eqref{Q1-GRP} as
\begin{equation}
\begin{array}{l}
\bu_t +\bbf(\bu)_x=- \bg(\bu)_y, \\[3mm]
\bu(x,\td y,0)=\left\{
\begin{array}{ll}
\bu_\ell(x,\td y), \ \ \ \ & x<0, \\
\bu_r(x,\td y) & x>0,
\end{array}
\right.
\end{array}
\label{Q2-GRP} 
\end{equation}
by fixing a $y$-coordinate.  That is, we solve 1-D GRP at a point $(0,\td y)$ on the interface, by considering the effect tangential to the interface $x=0$.  The value $\bg(\bu)_y$ at $(0,\td y)$  takes account of the local wave propagation. 
\vspace{0.2cm}

Again, the quasi 1-D GRP solver for solving \eqref{Q1-GRP}, particularly for the Euler equations \eqref{2d-Euler},  has the following two versions. The difference from 1-D version is that the multi-dimensional effect is included.  \vspace{0.2cm}

\subsubsection{Quasi-1-D acoustic case.}  At any point  $(0,\td y)$, if $\bu_\ell(0-0,\td y)\approx \bu_r(0+0,\td y)$ and $\|\nb\bu_\ell\|\neq \|\nb\bu_r\|$, we view it as a quai-1-D acoustic case. Denote $\bu_0: = \bu_\ell(0-0,\td y)\approx \bu_r(0+0,\td y)$ and $A(\bu_0) = \frac{\pt \bbf}{\pt \bu}(\bu_0)$.  We make the decomposition $A(\bu_0) =R \La R^{-1}$,
where  $\La=\mbox{diag}\{\la_i\}$, $R$ is the (left) eigenmatrix of $A(\bu_0)$.  Then the acoustic GRP solver takes 
\begin{equation}
\begin{array}{rl}
\d \left(\dfr{\pt \bu}{\pt t}\right)_{(0,\td y, 0)} =& \d  -R\La^+ R^{-1} \left(\dfr{\pt \bu_\ell}{\pt x}\right)_{(0-0,\td y)}-R I^+ R^{-1} \left(\dfr{\pt \bg(\bu_\ell)}{\pt y}\right)_{(0-0,\td y)}\\[3mm]
 &-R\La^- R^{-1} \left(\dfr{\pt \bu_r}{\pt x}\right)_{(0+0,\td y)}-R I^- R^{-1} \left(\dfr{\pt \bg(\bu_r)}{\pt y}\right)_{(0+0,\td y)},
\end{array}
\label{acoust}
\end{equation}
where $\La^+ =\mbox{diag}\{\max(\la_i,0)\}$, $\La^- =\mbox{diag}\{\min(\la_i,0)\}$,  $I^+ =\frac 12 \mbox{diag}\{1+\mbox{sign}(\la_i)\}$, $I^- =\frac 12 \mbox{diag}\{1-\mbox{sign}(\la_i)\}$.

\vspace{0.2cm}

\subsubsection{Quasi-1-D  nonlinear GRP solver.}  At any point $(0,\td y)$, if the difference $\|\bu_\ell(0-0,\td y)-\bu_r(0+0,\td y)\|$ is large, we regard it as the genuinely nonlinear case and have to solve the  quasi 1-D  GRP analytically. A key ingredient is how to understand $\bg(\bu)_y$. Here we construct the quasi 1-D GRP solver by two steps. 
\vspace{0.2cm}

(i) We solve the local 1-D planar Riemann problem for 
\begin{equation}
\begin{array}{l} 
\bv_t +\bbf(\bv)_x =0,\ \ \ \ t>0, \\[3mm] 
\bv(x,\td y,0) =\left\{
\begin{array}{ll}
\bu_\ell(0-0,\td y), \ \ \ &x<0,\\
\bu_r(0+0,\td y), & x>0,
\end{array}
\right.
\end{array}
\label{q1d}
\end{equation} 
to obtain the local Riemann solution $\bu_0 =\bv(0,\td y,0+0)$.  Just as in the acoustic case, we decompose $A(\bu_0) =\frac{\pt \bbf}{\pt \bu}(\bu_0) =R\La R^{-1}$.  Then we set 
\begin{equation}
\bh(x,\td y) = \left\{
\begin{array}{ll}
-R I^+ R^{-1} \left(\dfr{\pt \bg(\bu_\ell)}{\pt y}\right)_{(0-0,\td y)}, \ \  \  &x<0,\\[3mm]
-RI^- R^{-1} \left(\dfr{\pt \bg(\bu_r)}{\pt y}\right)_{(0+0,\td y)}, \ \  \  &x>0,
\end{array}
\right.
\end{equation} 
where $I^\pm$ are defined the same as in \eqref{acoust}. 
\vspace{0.2cm}

(ii)  We solve the quasi 1-D GRP 
\begin{equation}
\begin{array}{l}
\bu_t + \bbf(\bu)_x=h(x,\td y), \ \ \ \ t>0,\\[3mm]
\bu(x,\td y, 0) =\left\{
\begin{array}{ll}
\bu_\ell(x,\td y), \ \ \ &x<0,\\
\bu_r(x,\td y), & x>0,
\end{array}
\right.
\end{array}
\end{equation} 
to obtain $\frac{\pt \bu}{\pt t}(0,\td y, 0+0)$.  The details can be found in \cite{Li-2}.

\section*{Acknowlegement}

The authors appreciate Yue Wang and Kun Xu for their careful reading, which greatly improves the English presentation. Jiequan Li is supported by by NSFC (No. 11371063, 91130021),  the doctoral program from the Education Ministry of China (No. 20130003110004) and the Innovation Program from Beijing Normal University (No. 2012LZD08).

\vspace{0.5cm}
\end{document}

%% file: abbr-notation.tex
\def\bbf{\mathbf{f}}
\def\bg{\mathbf{g}}
\def\bh{\mathbf{h}}

\def\bn{\mathbf{n}}

\def\bu{\mathbf{u}}
\def\bv{\mathbf{v}}

\def\bx{\mathbf{x}}

\def\bF{\mathbf{F}}